\documentclass[10pt, leqno]{amsart}
\usepackage{amsfonts,amssymb}

\usepackage{amsmath}
\usepackage{amsthm}
\usepackage{mathtools}
\usepackage{amsrefs}
\usepackage[colorlinks=true,urlcolor=blue,linkcolor=blue,citecolor=blue]{hyperref}
\usepackage{doi}

\usepackage{qsymbols}
\usepackage{latexsym}
\usepackage{chngcntr}
\usepackage{xcolor}
\usepackage{paralist}
\usepackage{enumitem}

\renewcommand{\Re}{\re}

\newcommand{\fa}{\mathfrak{a}}

\newcommand{\ee}{\mathrm e}
\newcommand{\ii}{\mathrm i}
\newcommand{\id}{\mathrm{Id}}

\newcommand{\eps}{\varepsilon}
\DeclarePairedDelimiter{\abs}{\lvert}{\rvert}
\DeclarePairedDelimiter{\norm}{\lVert}{\rVert}

\newcommand{\cL}{\mathcal{L}}

\newcommand{\NN}{\mathbb{N}}

\newcommand{\RR}{\mathbb{R}}
\newcommand{\CC}{\mathbb{C}}

\newcommand{\1}{\mathbf{1}}

\renewcommand{\subset}{\subseteq}

\DeclareMathOperator*{\esssup}{ess\,sup}
\DeclareMathOperator*{\essinf}{ess\,inf}
\newcommand{\obs}{\mathrm{obs}}

\newcommand{\QQ}{\mathbb{Q}}

\renewcommand\Re{\operatorname{Re}}

\renewcommand{\d}[1]{\ensuremath\, {\operatorname{d}\!{#1}}}

\DeclareMathOperator{\Id}{Id}

{\bf}{\it}
{\bf}{\it}

\newtheorem{thm}{Theorem}[section]

\newtheorem{proposition}[thm]{Proposition}

\newtheorem{hypo}[thm]{Hypothesis}

\theoremstyle{definition}

\newtheorem{example}[thm]{Example}

\theoremstyle{remark}
\newtheorem{rem}[thm]{Remark}
\newtheorem{remark}[thm]{Remark}
\usepackage{changes}
\usepackage{tcolorbox}

\numberwithin{equation}{section} 
\usetikzlibrary{calc,intersections,arrows.meta}
\usetikzlibrary{decorations.pathreplacing,calligraphy}
  
\makeatletter
\def\Ddots{\mathinner{\mkern1mu\raise\p@
\vbox{\kern7\p@\hbox{.}}\mkern2mu
\raise4\p@\hbox{.}\mkern2mu\raise7\p@\hbox{.}\mkern1mu}}
\makeatother

\title[A unified observability result]{A unified observability result for non-autonomous observation problems}
\author{Fabian Gabel and Albrecht Seelmann}
\address{Institut f\"ur Mathematik, TU Hamburg, Am Schwarzenberg-Campus 3, 21073 Hamburg, Germany}
\email{\href{mailto:fabian.gabel@tuhh.de}{fabian.gabel@tuhh.de}}
\address{ Fakultät für Mathematik, TU Dortmund, Vogelpothsweg 87, 44227 Dortmund, Germany}
\email{\href{mailto:albrecht.seelmann@tu-dortmund.de}{albrecht.seelmann@tu-dortmund.de}}

\keywords{Banach space, evolution family, non-autonomous system, null-controllability, observability, uncertainty principle,
dissipation estimate, density point}
\subjclass[2020]{Primary 93B07; Secondary 35Q93, 47N70, 93B28}
\thanks{}
\dedicatory{}
\hypersetup{
pdftitle={},
pdfauthor={Fabian Gabel, Albrecht Seelmann},
pdfcreator={LaTeX with hyperref (gitREF, gitSHA)},
pdfkeywords={},
pdfsubject={MSC2020 Primary ; Secondary },
}

\begin{document}
\begin{abstract}
	A final-state observability result in the Banach space setting for non-autonomous observation problems is obtained that covers
	and extends all previously known results in this context, while providing a streamlined proof that follows the established
	Lebeau-Robbiano strategy.
\end{abstract}
\maketitle
%%%%%%%%%%%%%%%%%%%%%%%%%%%%%%%%%%%%%%%%%%%%%%%%%%%%%%%%%%%%%%%%%%%%%%%%%%%%%%%%%%%%%%%%%%%%%%%%%%%%%%%%%%%%%%%%%%%%%%%%%%%%%%%%%%%%%%%%%%%%%%%%%%%%%%%%%%%%%%%%%%%%

\section{Introduction}

Observability and null-controllability results for (non-)autonomous Cauchy problems are relevant especially in the field of
control theory of partial differential equations and have recently attracted a lot of attention in the literature. Here, the most
common approach towards \emph{final-state observability} is a so-called \emph{Lebeau-Robbiano strategy}, which combines a
suitable \emph{uncertainty principle} with a corresponding \emph{dissipation estimate} for the evolution family describing the
evolution of the system, see \eqref{eq:uncertaintyPrinciple} and \eqref{eq:dissipationEstimate} below, respectively. Certain
null-controllability results can then be inferred from final-state observability via a standard duality argument, see, e.g.,
\cite{BombachGST-22} for more information and also \cite{Gallaun_diss} for an holistic overview of duality theory for control
systems.

Such a Lebeau-Robbiano strategy has been considered, for instance, in \cite{Miller-10,BeauchardP,BeauchardEP-S-20,NakicTTV-20,
GallaunST-20,BombachGST-22,GallaunMS-23,KruseS-23}, see also \cite{EgidiNSTTV-20} for a review of other related results in this
context. The two most general results in this direction so far are \cite[Theorem~3.3]{BombachGST-22} and
\cite[Theorem~13]{BeauchardEP-S-20}, each highlighting different aspects and exhibiting certain advantages and disadvantages over
the other, both with regard to hypotheses and the asserted conclusion, see the discussion below. The aim of the present work is
to present a unified extension of both mentioned results, taking the best of each, thus allowing to apply the Lebeau-Robbiano
strategy to a broader range of observation problems and, at the same time, providing a streamlined proof.

\section{Lebeau-Robbiano strategy for non-autonomous observation problems}

For the reader's convenience, let us fix the following notational setup.

\begin{hypo}\label{hyp:observability}
	Let $X$ and $Y$ be Banach spaces, $T>0$, $E \subseteq [0,T]$ be measurable with positive Lebesgue measure, and
	$(U(t,s))_{0\leq s\leq t \leq T}$ be an exponentially bounded evolution family on $X$. Let $C\colon [0,T]\to \cL(X,Y)$ be
	essentially bounded on $E$ such that $[0,T] \ni t\mapsto \norm{C(t) U(t,0) x_0}_Y$ is measurable for all $x_0\in X$.
\end{hypo}
Here, we denote by $\cL(X,Y)$ the space of bounded operators from $X$ to $Y$. Also recall that
$(U(t,s))_{0\leq s\leq t \leq T} \subseteq \cL(X) \coloneqq \cL(X,X)$ is called an \emph{evolution family} of bounded operators on $X$
if
\begin{align}\label{eq:algebraic}
	U(s,s)
	=
	\Id
	\quad\text{ and }\quad 
  U(t,s)U(s,r)
  =
  U(t,r)
  \quad\text{ for }\
  0\leq r\leq s \leq t \leq T
  .
\end{align}
It is called \emph{exponentially bounded} if there exist $M \geq 1$ and $\omega \in \RR$ such that for all
$0 \leq s \leq t \leq T$ we have the bound $\norm{ U(t,s) }_{\cL(X)} \leq M \ee^{\omega (t - s)}$.

Evolution families are oftentimes used to describe the evolution of non-autonomous Cauchy problems, see, e.g.,
\cite[Section~2]{BombachGST-22} and the references cited therein. The family $(C(t))_{t\in[0,T]}$ in the mapping
$t\mapsto \norm{C(t) U(t,0) x_0}_Y$ can be understood as \emph{observation operators} through which the state of the system is
observed at each time $t \geq 0$. In the context of $L^p$-spaces, these are often chosen as multiplication operators by
characteristic functions for some (time-dependent) sensor sets, see, e.g., Example~\ref{ex:elliptic} below.

The following theorem now covers and extends all known previous results in this context, see the discussion below.

\begin{thm}\label{thm:observability}
	Assume Hypothesis~\ref{hyp:observability}. Let $(P_\lambda)_{\lambda>0}$ be a family in $\cL(X)$ such that for some constants
	$d_0,d_1,\gamma_1 > 0$ we have
	\begin{equation}\label{eq:uncertaintyPrinciple}\tag{essUCP}
		\forall \lambda > 0, \forall x\in X
		\colon
		\
		\norm{P_\lambda x}_{X}
		\leq
		d_0 \ee^{d_1 \lambda^{\gamma_1}} \essinf \bigl\{ E \ni \tau \mapsto \norm{ C(\tau)  P_\lambda x }_{Y} \bigr\}
		.
	\end{equation}
	Suppose also that for some constants $d_2\geq 1$, $d_3,\gamma_2,\gamma_3 >0$ with $\gamma_2 > \gamma_1$ and $\gamma_4 \geq 0$
	we have
	\begin{equation}\label{eq:dissipationEstimate}\tag{DE}
  	\begin{aligned}
			&\forall \lambda > 0, \forall 0\leq s <  t\leq T, \forall x\in X \colon\\
			&\norm{ (\id-P_\lambda) U(t,s) x }_{X} 
			\leq 
			d_2 \max\bigl\{ 1, (t - s)^{-\gamma_4} \bigr\} \ee^{-d_3 \lambda^{\gamma_2} (t-s)^{\gamma_3}} \norm{x}_{X}
			.
		\end{aligned}
	\end{equation}

	Then, there exists a constant $C_{\mathrm{obs}} > 0$ such that for each $r \in [1,\infty]$ and all $x_0\in X$ we have the
	final-state observability estimate
	\begin{equation}\label{eq:observabilityEstimate}\tag{OBS}
		\norm{ U(T,0) x_0 }_{X}
		\leq
		C_\obs
		\begin{cases}
			\Bigl( \int_E \norm{ C(t)U(t,0) x_0 }_Y^r \d t \Bigr)^{1/r}, & r < \infty,\\[1em]
			\esssup_{t\in E} \norm{ C(t)U(t,0)x_0 }_{Y}, & r=\infty.                
		\end{cases}
	\end{equation}
	Moreover, if for some interval $(\tau_1,\tau_2) \subseteq [0,T]$ with $\tau_1 < \tau_2$ we have
	$\abs{(\tau_1,\tau_2) \cap E} = \tau_2 - \tau_1$, then, depending on the value of $r$, the constant~$C_{\mathrm{obs}}$ can be
	bounded as
	\begin{equation}\label{eq:observabilityConstant}
		C_\obs
		\leq
		\frac{C_1}{(\tau_2-\tau_1)^{1/r}} \exp\Biggl( \frac{C_2}{(\tau_2-\tau_1)^{\frac{\gamma_1\gamma_3}{\gamma_2-\gamma_1}}}
			+ C_3T \Biggr)
		,
	\end{equation}
	with the usual convention $1/r = 0$ if $r=\infty$, where $C_1,C_2,C_3 > 0$ are constants not depending on $r$, $T$, $E$,
	$\tau_1$, or $\tau_2$.
\end{thm}

The above theorem represents the established Lebeau-Robbiano strategy, in which an uncertainty principle
\eqref{eq:uncertaintyPrinciple}, here with respect to the given family $(P_\lambda)_{\lambda>0}$ and uniform only on the subset
$E \subseteq [0,T]$, and a corresponding dissipation estimate in the form \eqref{eq:dissipationEstimate} are used as an input; it
should be emphasized that the requirement $\gamma_1 < \gamma_2$ is essential here. The output in the form of
\eqref{eq:observabilityEstimate} then constitutes a so-called \emph{final-state observability estimate} for the evolution family
$(U(t,s))_{0\leq s\leq t \leq T} \subseteq \cL(X)$ with respect to the family $(C(t))_{t\in[0,T]}$ of observation operators. The
corresponding constant $C_\obs$ in \eqref{eq:observabilityEstimate} is called \emph{observability constant}. An explicit form of
the constants $C_1,C_2,C_3$ in \eqref{eq:observabilityConstant} is given in Remark~\ref{rem:explicitConstant} below for easier
reference.

\subsection*{Discussion and extensions}

We first comment on two minor extensions of Theorem~\ref{thm:observability}.

\begin{rem}
	(1)
	It becomes clear from the proof, see \eqref{eq:blowup} below, that instead of the polynomial blow-up in the dissipation estimate
	\eqref{eq:dissipationEstimate} for small differences $t-s$ one can also allow a certain \emph{(sub-)exponential} blow-up. More
	precisely, one may replace the term $\max\{1,(t-s)^{-\gamma_4}\}$ by a factor of the form
	$\exp(c(t-s)^{-\frac{\gamma_1\gamma_3}{\gamma_2-\gamma_1}})$ with some constant $c > 0$.

	(2)
	If $\abs{(\tau_1,\tau_2) \cap E} = \tau_2 - \tau_1$ for some interval $(\tau_1,\tau_2) \subseteq [0,T]$ with $\tau_1 < \tau_2$,
	then the dissipation estimate \eqref{eq:dissipationEstimate} is actually needed only for $\tau_1 < s < t \leq \tau_2$, cf.\ part
	(1) of Remark~\ref{rem:proof} below.

\end{rem}

Let us now compare Theorem~\ref{thm:observability} to earlier results in the literature.

\begin{rem}
	(1)
	In the particular case where $[\tau_1,\tau_2] = [0,T]$, the bound on $C_\obs$ in \eqref{eq:observabilityConstant} is completely
	consistent, except perhaps for some minor differences in the explicit form of the constants $C_1,C_2,C_3$ in
	Remark~\ref{rem:explicitConstant} below, with all bounds obtained earlier for $E = [0,T]$ in \cites{NakicTTV-20,GallaunST-20,
	BombachGST-20} in the autonomous case (i.e.\ $C(t) \equiv C$ and the evolution family being actually a semigroup) and in
	\cite{BombachGST-22} in the non-autonomous case.

	(2)
	Theorem~\ref{thm:observability} covers \cite[Theorem~3.3]{BombachGST-22}, while allowing a polynomial blow-up for small
	differences $t - s$ in the dissipation estimate \eqref{eq:dissipationEstimate}. Such a blow-up has first been considered in
	\cite[Theorem~13]{BeauchardEP-S-20}, but under much more restrictive assumptions, see item (4) below. Moreover, in contrast to
	\cite[Theorem~3.3]{BombachGST-22}, Theorem~\ref{thm:observability} requires the uncertainty relation
	\eqref{eq:uncertaintyPrinciple} only on a subset of $[0,T]$ of positive measure, instead of the whole interval $[0,T]$, and thus
	allows far more general families of observation operators. These families also need to be uniformly bounded only on this
	measurable subset and not on the whole interval $[0,T]$.

	(3)
	The results from \cite[Theorem A.1]{BombachGST-20} and \cite[Theorem 2.1]{GallaunST-20} formulate a variant of our
	Theorem~\ref{thm:observability} in the autonomous case with $E = [0,T]$, but assume \eqref{eq:uncertaintyPrinciple} and
	\eqref{eq:dissipationEstimate} only for $\lambda > \lambda^*$ with some $\lambda^* \geq 0$. Our current formulation with the
	whole range $\lambda > 0$, just as in \cite[Theorem~3.3]{BombachGST-22}, is not really a restriction to that. Indeed, by a
	change of variable, one may then simply consider the family $(P_{\lambda+\lambda^*})_{\lambda>0}$ instead, with a
	straightforward adaptation of the parameters $d_0$ and $d_1$ in \eqref{eq:uncertaintyPrinciple}. In this sense,
	Theorem~\ref{thm:observability} completely covers the results from \cite{BombachGST-20} and \cite{GallaunST-20}.

	(4)
	Let $\tau_1 \in [0,T)$ be such that $\abs{ (\tau_1,T) \cap E } > 0$. By a change of variable, namely via considering
	$C(\cdot + \tau_1)$ on $[0,T-\tau_1]$ and the evolution family $U(t+\tau_1,s+\tau_1)_{0\leq s<t\leq T-\tau_1}$, one may replace
	$U(T,0)$, $U(t,0)$, and $E$ in \eqref{eq:observabilityEstimate} by $U(T,\tau_1)$, $U(t,\tau_1)$, and $(\tau_1,T) \cap E$
	respectively; note that \eqref{eq:uncertaintyPrinciple} and \eqref{eq:dissipationEstimate} then remain valid with the same
	constants. In this sense, Theorem~\ref{thm:observability} entirely covers \cite[Theorem~13]{BeauchardEP-S-20}, while leaving
	the Hilbert space setting and not requiring strong continuity or contractivity of the evolution family. At the same time, our
	bound on $C_\obs$ in \eqref{eq:observabilityConstant} contains an additional prefactor $1/(\tau_2-\tau_1)^{1/r}$ in front of the
	exponential term, which significantly changes the asymptotics of the estimate as $\tau_2-\tau_1$ (and thus also $T$) gets large.
	Such improved asymptotics has proved extremely useful in the past, for instance, when considering homogenization limits as in
	\cite{NakicTTV-20}.

\end{rem}

In order to support the above comparison, we briefly revisit \cite[Theorem~4.8]{BombachGST-22} in the following example.

\begin{example}\label{ex:elliptic}
	Let $\fa$ be a uniformly strongly elliptic polynomial of degree $m \geq 2$ in $\RR^d$ with coefficients
	$a_\alpha \in L^\infty([0,T])$, that is, $\fa \colon [0,T] \times \RR^d \to \CC$ with
	\[
		\fa(t,\xi)
		=
		\sum_{\substack{\abs{\alpha}\leq m}} a_\alpha(t) (\ii\xi)^\alpha
		,\quad
		t \in [0,T]
		,\
		\xi \in \RR^d
		,
	\]
	such that for some $c > 0$ we have
	\[
		\Re \sum_{\abs{\alpha} = m} a_\alpha(t,\xi)
		\geq
		c\abs{\xi}^m
		,\quad
		t \in [0,T]
		,\
		\xi \in \RR^d
		.
	\]
	Let $p \in [1,\infty]$. It was shown in \cite[Theorem~4.4]{BombachGST-22} that there is an exponentially bounded evolution
	family $(U_p(t,s))_{0\leq s\leq t\leq T}$ in $L^p(\RR^d)$ associated to $\fa$. Let $(\Omega(t))_{t\in[0,T]}$ be
	a family of measurable subsets of $\RR^d$ such that the mapping $[0,T] \times \RR^d \ni (t,x) \mapsto \1_{\Omega(t)}(x)$ is
	measurable. Then, $\norm{ \1_{\Omega(\cdot)}U_p(\cdot,0)u_0 }_{L^p(\RR^d)}$ is for all $u_0 \in L^p(\RR^d)$ measurable on
	$[0,T]$ by \cite[Lemma~4.7]{BombachGST-22}, so that Hypothesis~\ref{hyp:observability} with $C(t) = \1_{\Omega(t)}$ is satisfied
	for every choice of measurable $E \subseteq [0,T]$ with positive measure. Moreover, a dissipation estimate as in
	\eqref{eq:dissipationEstimate} with $\gamma_2 = m$ (but without blow-up, i.e.\ $\gamma_4 = 0$) was established in the proof of
	\cite[Theorem~4.8]{BombachGST-22} with $P_\lambda$ being some smooth frequency cutoffs. It remains to consider a corresponding
	essential uncertainty principle \eqref{eq:uncertaintyPrinciple}.

	Suppose that the family $(\Omega(t))_{t\in[0,T]}$ of subsets is \emph{uniformly thick} on $E$ in the sense that there are
	$L,\rho > 0$ such that for all $x \in \RR^d$ and all $t \in E$ we have
	\begin{equation}\label{eq:uniformThick}
		\abs{ \Omega(t) \cap ((0,L)^d + x) }
		\geq
		\rho L^d
		.
	\end{equation}
	Following the proof of \cite[Theorem~4.8]{BombachGST-22}, we see that an essential uncertainty principle as in
	\eqref{eq:uncertaintyPrinciple} holds with $\gamma_1 = 1 < \gamma_2$, so that Theorem~\ref{thm:observability} can be applied.
	Here, the set $E$ for which \eqref{eq:uniformThick} needs to hold could be any measurable subset of $[0,T]$ with positive
	measure, for instance, $E = [0,T] \setminus \QQ$ (satisfying $\abs{E} = T$) or even some fractal set, say of
	Cantor-Smith-Volterra	type. In particular, this allows for completely arbitrary choices of measurable $\Omega(t)$ for
	$t \notin E$, even $\Omega(t) = \emptyset$. By contrast, these choices for $\Omega(t)$ would not be allowed in
	\cite[Theorem~4.8]{BombachGST-22}, where \eqref{eq:uniformThick} is required to hold for all $x \in \RR^d$ and \emph{all}
	$t \in [0,T]$ and is thus much more restrictive on the choice of $(\Omega(t))_{t\in[0,T]}$.
\end{example}

\begin{rem}
	In the situation of Example~\ref{ex:elliptic} with $p < \infty$, it was shown in \cite[Theorem~4.10]{BombachGST-22} that an
	observability estimate as in \eqref{eq:observabilityEstimate} can hold with $r < \infty$ only if the family
	$(\Omega(t))_{t\in[0,T]}$ is \emph{mean thick} in the sense that for some $L,\rho > 0$ we have
	\[
		\frac{1}{T} \int_0^T \abs{ \Omega(t) \cap ((0,L)^d + x) } \d t
		\geq
		\rho L^d
		\quad\text{ for all }\
		x \in \RR^d
		.
	\]
	It is easy to see that families which are uniformly thick on a subset of $[0,T]$ of positive measure as in~\eqref{eq:uniformThick} are also mean thick in the
	above sense (with possibly different parameters), but the converse need not be true. A corresponding example in $\RR$ is the
	family $(\Omega(t))_{t\in[0,T]}$ with $\Omega(t) = (0,\infty)$ for $t \leq T/2$ and $\Omega(t) = (-\infty,0)$ for
	$T/2 < t \leq T$. It is yet unclear whether such choices also lead to an observability estimate as in
	\eqref{eq:observabilityEstimate} or anything similar. In this sense, Example~\ref{ex:elliptic} and
	Theorem~\ref{thm:observability} still leave a gap between necessary and sufficient conditions on the family
	$(\Omega(t))_{t\in[0,T]}$ towards final-state observability.
\end{rem}

\section{Proof of Theorem~\ref{thm:observability}}

Our proof of Theorem~\ref{thm:observability} is a streamlined adaptation of earlier approaches, especially of that from
\cite{BombachGST-22} and its predecessors \cite{NakicTTV-20}, \cite{GallaunST-20}, and \cite{BeauchardEP-S-20}. It avoids the
interpolation argument in \cite{BombachGST-22} and is thus much more direct and, at the same time, requires an uncertainty
relation only on a measurable subset of $[0,T]$ of positive measure.

\begin{proof}[Proof of Theorem~\ref{thm:observability}]
	Let us fix $x_0 \in X$. For $0  \leq t \leq T$ we abbreviate
	\[
		F(t) \coloneqq \norm{ U(t,0)x_0 }_X
		,\quad
		G(t) \coloneqq \norm{ C(t) U(t,0)x_0 }_Y
		.
	\]
	By H\"older's inequality we clearly have 
	\begin{equation}\label{eq:Hoelder}
		\norm{ G(\cdot) }_{L^1(E)}
		\leq
		\abs{E}^{1-\frac{1}{r}} \norm{ G(\cdot) }_{L^r(E)}
	\end{equation}
	with the usual convention $1/\infty = 0$. Hence, estimate \eqref{eq:observabilityEstimate} for $r > 1$ follows from the one for
	$r = 1$ by multiplying the corresponding constant $C_\obs$ by $\abs{E}^{1-\frac{1}{r}} \leq \max\{ 1 , \abs{E} \}$. It therefore
	suffices to show \eqref{eq:observabilityEstimate} for $r = 1$, which in the new notation reads
	\begin{equation}\label{eq:final-stateFG}
  	F(T)
  	\leq
  	C_\obs \int_E G(t) \d t
  	.
	\end{equation}

	Upon possibly removing from $E$ a set of measure zero, we may assume without loss of generality that
	\eqref{eq:uncertaintyPrinciple} holds with $\essinf$ replaced by $\inf$ and that $C(\cdot)$ is uniformly bounded on $E$. Let us
	then show that there exist constants $c_1,c_2 > 0$ such that for all $0 \leq s < t \leq T$ with $t \in E$ and all
	$\eps \in (0,1)$ we have
	\begin{equation}\label{eq:epsilonBalance}
		F(t) 
		\leq
		c_1 \exp\biggl( c_2 \Bigl( \frac{1}{t-s} \Bigr)^{\frac{\gamma_1\gamma_3}{\gamma_2-\gamma_1}} \biggr)
			\bigl( \varepsilon^{-1} G(t) + \varepsilon F(s) \bigr)
		.
	\end{equation}
	To this end, let $\eps \in (0,1)$ and fix $0 \leq s < t \leq T$ with $t \in E$. For $\lambda > 0$ we introduce
	\[
		F_\lambda
		\coloneqq
		\norm{ P_\lambda U(t,0)x_0 }_X
		,\quad
		F_\lambda^\perp
		\coloneqq
		\norm{ (\id - P_\lambda)U(t,0)x_0 }_X
		,
	\]
	as well as
	\[
		G_\lambda
		\coloneqq
		\norm{ C(t)P_\lambda U(t,0)x_0 }_Y
		,\quad
		G_\lambda^\perp
		\coloneqq
		\norm{ C(t)(\id - P_\lambda)U(t,0)x_0 }_Y
		.
	\]
	The uncertainty relation \eqref{eq:uncertaintyPrinciple} and the uniform boundedness of $C(\cdot)$ on $E$ then give
	\[
		F_\lambda
		\leq
		d_0 \ee^{d_1 \lambda^{\gamma_1}} G_\lambda
		\quad\text{ and }\quad
		G_\lambda^\perp
		\leq
		\norm{ C(\cdot) }_{E,\infty} F_\lambda^\perp
		\quad\text{ for all }\
		\lambda > 0
		,
	\]
	where $\norm{C(\cdot)}_{E,\infty} \coloneqq \sup_{\tau \in E} \norm{C(\tau)}_{\cL(X,Y)} < \infty$. Since by triangle inequality
	$F(t) \leq F_\lambda + F_\lambda^\perp$ and $G_\lambda \leq G(t) + G_\lambda^\perp$, the latter implies that for all
	$\lambda > 0$ we have
	\begin{equation}\label{eq:F}
		F(t)
		\leq
		d_0 \ee^{d_1\lambda^{\gamma_1}} \bigl( G(t) + \norm{C(\cdot)}_{E,\infty} F_\lambda^\perp \bigr) + F_\lambda^\perp
		\leq
		\ee^{d_1\lambda^{\gamma_1}} \bigl( d_0 G(t) + (d_0\norm{C(\cdot)}_{E,\infty} + 1)F_\lambda^\perp \bigr)
		.
	\end{equation}
	Now, writing $U(t,0)x_0 = U(t,s)U(s,0)x_0$ by \eqref{eq:algebraic}, we obtain from \eqref{eq:dissipationEstimate} with
	$x = U(s,0)x_0$ that
	\[
		F_\lambda^\perp
		\leq
		d_2 \max\{ 1 , (t-s)^{-\gamma_4} \} \ee^{-d_3\lambda^{\gamma_2}(t-s)^{\gamma_3}} F(s)
		,
	\]
	and inserting this into the preceding estimate \eqref{eq:F} yields for all $\lambda > 0$ that
	\begin{equation}\label{eq:F2}
		\begin{aligned}
			F(t)
			&\leq
			c_1 \ee^{d_1\lambda^{\gamma_1}} \max\{ 1 , (t-s)^{-\gamma_4} \}
				\bigl( G(t) + \ee^{-d_3\lambda^{\gamma_2}(t-s)^{\gamma_3}} F(s) \bigr)\\
			&=
			c_1 \ee^{f(\lambda)} \max\{ 1 , (t-s)^{-\gamma_4} \} \bigl( \ee^{\frac{d_3}{2}\lambda^{\gamma_2}(t-s)^{\gamma_3}}G(t) +
				\ee^{-\frac{d_3}{2}\lambda^{\gamma_2}(t-s)^{\gamma_3}} F(s) \bigr)
		\end{aligned}
	\end{equation}
	with
	\begin{equation}\label{eq:c1}
		c_1
		\coloneqq
		\max\{ d_0 , (d_0\norm{C(\cdot)}_{E,\infty}+1)d_2 \}
		\geq
		1
		\quad\text{ and }\quad
		f(\lambda)
		\coloneqq
		d_1\lambda^{\gamma_1} - \frac{d_3}{2}\lambda^{\gamma_2}(t-s)^{\gamma_3}
		.
	\end{equation}
	Let us maximize $f(\lambda)$ with respect to $\lambda$. In light of $\gamma_2 > \gamma_1$ by hypothesis, a straightforward
	calculation reveals that $f$ takes its maximal value on $(0,\infty)$ at the point
	\[
		\lambda^* 
		\coloneqq 
		\Bigr( \frac{2d_1\gamma_1}{d_3\gamma_2} \Bigr)^{\frac{1}{\gamma_2-\gamma_1}}
			\Bigl( \frac{1}{t-s} \Bigr)^{\frac{\gamma_3}{\gamma_2-\gamma_1}}
		>
		0
		.
	\]
	Taking into account the relation $\frac{\gamma_1}{\gamma_2-\gamma_1} + 1 = \frac{\gamma_2}{\gamma_2-\gamma_1}$, we observe that
	\begin{align*}
		\frac{d_3}{2} (\lambda^*)^{\gamma_2} (t-s)^{\gamma_3}
		&=
		\frac{d_3}{2} \Bigl( \frac{2d_1\gamma_1}{d_3\gamma_2} \Bigr)^{\frac{\gamma_2}{\gamma_2-\gamma_1}}
			\Bigl( \frac{1}{t-s} \Bigr)^{\frac{\gamma_2\gamma_3}{\gamma_2-\gamma_1}-\gamma_3}\\
		&=
		\frac{d_1\gamma_1}{\gamma_2} \Bigl( \frac{2d_1\gamma_1}{d_3\gamma_2} \Bigr)^{\frac{\gamma_1}{\gamma_2-\gamma_1}}
			\Bigl( \frac{1}{t-s} \Bigr)^{\frac{\gamma_1\gamma_3}{\gamma_2-\gamma_1}}
			=
			\frac{d_1\gamma_1}{\gamma_2} (\lambda^*)^{\gamma_1}
		.
	\end{align*}
	We may therefore estimate $f(\lambda)$ as
	\[
		f(\lambda)
		\leq
		f(\lambda^*)
		=
		d_1 \Bigl( 1-\frac{\gamma_1}{\gamma_2} \Bigr) (\lambda^*)^{\gamma_1}
		=
		d_1 \Bigl( 1-\frac{\gamma_1}{\gamma_2} \Bigr) \Bigl( \frac{2d_1\gamma_1}{d_3\gamma_2}
			\Bigr)^{\frac{\gamma_1}{\gamma_2-\gamma_1}} \Bigl( \frac{1}{t-s} \Bigr)^{\frac{\gamma_1\gamma_3}{\gamma_2-\gamma_1}}
		.
	\]
	Moreover, using the elementary bound $\xi^\alpha \leq \ee^{\alpha \xi}$ for $\alpha,\xi > 0$, we have
	\begin{equation}\label{eq:blowup}
		\max\{ 1, (t - s)^{-\gamma_4} \}
		\leq
		\exp\biggl( \frac{\gamma_4 (\gamma_2 - \gamma_1)}{\gamma_1\gamma_3} \Bigl(\frac{1}{t - s}
			\Bigr)^{\frac{\gamma_1\gamma_3}{\gamma_2 - \gamma_1}} \biggr)
		.
	\end{equation}
	Inserting this and the preceding bound on $f(\lambda)$ into \eqref{eq:F2}, we arrive for all $\lambda > 0$ at
	\[
		F(t) 
    \leq
    c_1 \exp\biggl( c_2 \Bigl(\frac{1}{t-s}\Bigr)^{\frac{\gamma_1\gamma_3}{\gamma_2-\gamma_1}} \biggr)
    	\Bigl( \ee^{\frac{d_3}{2} \lambda^{\gamma_2}(t-s)^{\gamma_3}} G(t) + \ee^{-\frac{d_3}{2} \lambda^{\gamma_2}(t-s)^{\gamma_3}}
    	F(s)\Bigr)
	\]
	with
	\begin{equation}\label{eq:c2}
		c_2
		\coloneqq 
		d_1 \Bigl(1-\frac{\gamma_1}{\gamma_2}\Bigr)\Bigl( \frac{2 d_1\gamma_1}{d_3\gamma_2} \Bigr)^{\frac{\gamma_1}{\gamma_2-\gamma_1}}
			+ \frac{\gamma_4 ( \gamma_2 - \gamma_1)}{\gamma_1 \gamma_3}
		.
	\end{equation}
	We finally choose $\lambda > 0$ such that $\eps = \ee^{-\frac{d_3}{2}\lambda^{\gamma_2}(t-s)^{\gamma_3}}$, which shows that
	\eqref{eq:epsilonBalance} is valid; note that indeed neither $c_1$ nor $c_2$ depend on $s$ or $t$.

	Let $q \coloneqq \bigl(\frac{3}{4}\bigr)^{\frac{\gamma_2-\gamma_1}{\gamma_1\gamma_3}} < 1$, and choose by Lebesgue's
	differentiation theorem a (right) density point $\ell \in [0,T) \cap E$ of $E$ in the sense of
	Appendix~\ref{sec:approxDensityPoint}. Proposition~\ref{prop:phungw-13} then guarantees that there is a strictly decreasing
	sequence $(\ell_m)_{m\in\NN}$ in $(\ell,T]$ of the form $\ell_m = \ell + q^{m-1}(\ell_1-\ell)$, $m \in \NN$, satisfying
	\begin{equation}\label{eq:phungw-13:Proportion}
		\abs{(\ell_{m+1},\ell_m) \cap E}
		\geq
		\frac{\delta_m}{3}
		,\quad
		m \in \NN
		,
	\end{equation}
	where $\delta_{m} \coloneqq \ell_m - \ell_{m+1}$, $m \in \mathbb{N}$. It is also easy to see that
	\begin{equation}\label{eq:phungw-13:Step}
		\delta_{m + 1}
		=
		q\delta_{m}
		,\quad
		m \in \NN
		.
	\end{equation}
	Since the evolution family $(U(t,s))_{0\leq s\leq t\leq T}$ is exponentially bounded by hypothesis, there	exist $M \geq 1$ and
	$\omega \in \RR$ such that
	\[
		F(t)
		=
		\norm{ U(t,s)U(s,0)x_0 }_X
		\leq
		M\ee^{\omega(t-s)} F(s)
		\quad\text{ for all }\
		0\leq s\leq t\leq T
		.
	\]
	Setting $\omega_+ \coloneqq \max\{\omega,0\}$, this in particular implies for each $m \in \NN$ and all
	$t \in (\ell_{m+1},\ell_m)$ that
	\begin{equation}\label{eq:FExpBounded}
		F(\ell_m)
		\leq
		M\ee^{\omega(\ell_m-t)} F(t)
		\leq
		M\ee^{\omega_+\delta_1} F(t)
		.
	\end{equation}
	Define
	\[
		\xi_m
		\coloneqq
		\ell_{m+1} + \frac{\delta_m}{6}
		\in
		(\ell_{m+1},\ell_m)
		,\quad
		m \in \NN
		,
	\]
	which in light of \eqref{eq:phungw-13:Proportion} satisfies
	\begin{equation}\label{eq:phungwEstimateXi}
		\abs{ (\xi_m,\ell_m) \cap E }
		=
		\abs{ (\ell_{m+1},\ell_m) \cap E } - \abs{ (\ell_{m+1},\xi_m) \cap E }
		\geq
		\frac{\delta_m}{3} - (\xi_m - \ell_{m+1})
		=
		\frac{\delta_m}{6}
		>
		0
		.
	\end{equation}
	Combining \eqref{eq:FExpBounded} and \eqref{eq:epsilonBalance} with $s = \ell_{m+1}$, we obtain for all $m \in \NN$,
	$t \in (\xi_m,\ell_m) \cap E$, and $\eps \in (0,1)$ that
	\[
		F(\ell_m)
		\leq
		c_3 \exp\Biggl( \frac{c_4}{\delta_m^{\frac{\gamma_1\gamma_3}{\gamma_2-\gamma_1}}} \Biggr)
			\bigl( \varepsilon^{-1} G(t) + \varepsilon F(\ell_{m+1}) \bigr)
	\]
	with
	\[
		c_3
		\coloneqq
		Mc_1\ee^{\omega_+\delta_1}
		\geq
		1
		\quad\text{ and }\quad
		c_4
		\coloneqq
		c_2 6^{\frac{\gamma_1\gamma_3}{\gamma_2-\gamma_1}}
		>
		0
		,
	\]
	where we have taken into account that $t-\ell_{m+1} \geq \xi_m-\ell_{m+1} = \delta_m / 6$. With the particular choice
	\[
		\eps
		=
		c_3^{-1} q \exp\Biggl( -\frac{2c_4}{\delta_m^{\frac{\gamma_1\gamma_3}{\gamma_2-\gamma_1}}} \Biggr)
		\in
		(0,1)
		,
	\]
	the latter turns into
	\begin{equation}\label{eq:telescopeBeforeRearrange}
		F(\ell_m)
		\leq
		c_3^2 q^{-1} \exp\Biggl( \frac{3c_4}{\delta_m^{\frac{\gamma_1\gamma_3}{\gamma_2-\gamma_1}}} \Biggr) G(t)
			+ q \exp\Biggl( -\frac{c_4}{\delta_m^{\frac{\gamma_1\gamma_3}{\gamma_2-\gamma_1}}} \Biggr) F(\ell_{m+1})
		.
	\end{equation}
	Observing that
	\[
		\exp\Biggl( \frac{4c_4}{\delta_m^{\frac{\gamma_1\gamma_3}{\gamma_2-\gamma_1}}} \Biggr)
		=
		\exp\Biggl( \frac{3c_4}{\delta_{m+1}^{\frac{\gamma_1\gamma_3}{\gamma_2-\gamma_1}}} \Biggr)
	\]
	by \eqref{eq:phungw-13:Step} and the choice of $q$, multiplying \eqref{eq:telescopeBeforeRearrange} by $\delta_m$ and
	rearranging terms yields
	\[
		\delta_m \exp\Biggl( -\frac{3c_4}{\delta_m^{\frac{\gamma_1\gamma_3}{\gamma_2-\gamma_1}}} \Biggr) F(\ell_m)
		-
		\delta_{m+1} \exp\Biggl( -\frac{3c_4}{\delta_{m+1}^{\frac{\gamma_1\gamma_3}{\gamma_2-\gamma_1}}} \Biggr) F(\ell_{m+1})
		\leq
		c_3^2 q^{-1} \delta_m G(t)
	\]
	for all $t \in (\xi_m,\ell_m) \cap E$, $m \in \NN$. Taking into account \eqref{eq:phungwEstimateXi}, integrating the latter with
	respect to $t \in (\xi_m,\ell_m) \cap E$ leads to
	\[
		\delta_m \exp\Biggl( -\frac{3c_4}{\delta_m^{\frac{\gamma_1\gamma_3}{\gamma_2-\gamma_1}}} \Biggr) F(\ell_m)
		-
		\delta_{m+1} \exp\Biggl( -\frac{3c_4}{\delta_{m+1}^{\frac{\gamma_1\gamma_3}{\gamma_2-\gamma_1}}} \Biggr) F(\ell_{m+1})
		\leq
		6c_3^2 q^{-1} \int_{\ell_{m+1}}^{\ell_m} \1_E(t) G(t) \d t
	\]
	for all $m \in \NN$. Note here that the exponential boundedness of the evolution family guarantees that the sequence
	$(F(\ell_m))_{m\in \NN}$ is bounded. Since also $\delta_m \to 0$ and $\ell_m \to \ell$ as $m \to \infty$, summing the last
	inequality over all $m \in \NN$ implies by a telescoping sum argument that
	\[
		\delta_1 \exp\Biggl( -\frac{3c_4}{\delta_1^{\frac{\gamma_1\gamma_3}{\gamma_2-\gamma_1}}} \Biggr) F(\ell_1)
		\leq
		6c_3^2 q^{-1} \int_{\ell}^{\ell_1} \1_E(t) G(t) \d t
		\leq
		6c_3^2 q^{-1} \int_E G(t) \d t
		,
	\]
	which can be rewritten as
	\[
		F(\ell_1)
		\leq
		6c_3^2 q^{-1} \delta_1^{-1} \exp\Biggl( \frac{3c_4}{\delta_1^{\frac{\gamma_1\gamma_3}{\gamma_2-\gamma_1}}} \Biggr)
			\int_E G(t) \d t
		.
	\]
	Now, we have $F(T) \leq M\ee^{\omega(T-\ell_1)}F(\ell_1)$ by using once more the exponential boundedness of the evolution
	family, which shows that \eqref{eq:final-stateFG} holds with
	\begin{equation}\label{eq:observabilityConstantExplicit}
		C_\obs
		=
		6c_3^2 q^{-1} \delta_1^{-1} \exp\Biggl( \frac{3c_4}{\delta_1^{\frac{\gamma_1\gamma_3}{\gamma_2-\gamma_1}}} \Biggr)
			M\ee^{\omega(T-\ell_1)}
		.
	\end{equation}

	Finally, suppose that $\abs{ (\tau_1,\tau_2) \cap E } = \tau_2 - \tau_1$ for some interval $(\tau_1,\tau_2) \subseteq [0,T]$
	with $\tau_1 < \tau_2$. We may then simply choose $\ell = \tau_1$ and $\ell_1 = \tau_2$ in the above reasoning, leading to
	$\delta_1 = (1-q)(\tau_2-\tau_1)$.  For $r \geq 1$, in light of \eqref{eq:Hoelder} with $E$ replaced by
	$(\tau_1,\tau_2) \cap E$, we conclude that $C_\obs$ in \eqref{eq:observabilityEstimate} can be bounded as in
	\eqref{eq:observabilityConstant}, which completes the proof.
\end{proof}%

For organizational purposes, we extract from the above proof the following more explicit bound on the observability constant.

\begin{rem}\label{rem:explicitConstant}
	In the case where $\abs{ (\tau_1,\tau_2) \cap E } = \tau_2 - \tau_1$ for some interval $(\tau_1,\tau_2) \subseteq [0,T]$ with
	$\tau_1 < \tau_2$ in (the proof of) Theorem~\ref{thm:observability}, we actually have
	$\abs{(\ell_{m+1},\ell_m) \cap E} = \delta_m$ instead of the weaker \eqref{eq:phungw-13:Proportion}. Consequently, one may
	choose $\xi_m \coloneqq \ell_{m+1} + \delta_m/2$, resulting in $\abs{(\xi_m,\ell_m) \cap E} = \delta_m/2$ instead of
	\eqref{eq:phungwEstimateXi}. We may therefore replace the numerical factor $6$ in \eqref{eq:observabilityConstantExplicit} and
	the constant $c_4$ by $2$, so that
	\[
		C_\obs
		\leq
		\frac{C_1}{(\tau_2-\tau_1)^{1/r}} \exp\Biggl( \frac{C_2}{(\tau_2-\tau_1)^{\frac{\gamma_1\gamma_3}{\gamma_2-\gamma_1}}}
			+ C_3T \Biggr)
	\]
	with
	\[
		C_1
		\coloneqq
		\frac{2M^3c_1^2}{q(1-q)}
		,\quad
		C_2
		\coloneqq
		3c_2 \Bigl( \frac{2}{1-q} \Bigr)^{\frac{\gamma_1\gamma_3}{\gamma_2-\gamma_1}}
		,\quad
		C_3
		\coloneqq
		3\omega_+
		,
	\]
	where
	\[
	  c_1
		=
		\max\{ d_0 , (d_0\norm{C(\cdot)}_{E,\infty}+1)d_2 \}
		\geq
		1
		,\quad
		c_2
		=
		d_1 \Bigl(1-\frac{\gamma_1}{\gamma_2}\Bigr)\Bigl( \frac{2 d_1\gamma_1}{d_3\gamma_2} \Bigr)^{\frac{\gamma_1}{\gamma_2-\gamma_1}}
			+ \frac{\gamma_4 ( \gamma_2 - \gamma_1)}{\gamma_1 \gamma_3}
	\]
	are as in \eqref{eq:c1} and \eqref{eq:c2}, respectively, and where $M \geq 1$ and
	$\omega_+ = \max\{ \omega , 0 \}$, $\omega \in \RR$, are such that $\norm{ U(t,s) }_{\cL(X)} \leq M\ee^{\omega(t-s)}$ for all
	$0 \leq s \leq t \leq T$.
\end{rem}

\begin{rem}\label{rem:proof}
	(1)
	The proof of Theorem~\ref{thm:observability} actually requires the dissipation estimate \eqref{eq:dissipationEstimate} only for
	$\ell \leq s < t \leq \ell_1$.

	(2)
	It is worth to emphasize that in contrast to \cite[Theorem~13]{BeauchardEP-S-20}, our proof of Theorem~\ref{thm:observability}
	does \emph{not} require the sequence $(\ell_m)_{m\in\NN}$ to belong to the set $E$, but only to
	$(\ell,T]$. This relaxed requirement is much easier to satisfy than the stricter one in \cite[Proposition~14]{BeauchardEP-S-20}
	and is reviewed in Proposition~\ref{prop:phungw-13} in the appendix.
\end{rem}

\appendix

\section{Approximating density points of measurable subsets of the real line}\label{sec:approxDensityPoint}

Recall that a point $\ell \in \RR$ is called a \emph{right density point} (resp.\ \emph{density point}) of a measurable set
$E \subset \RR$ if
\begin{equation}\label{eq:lebesguePoint}
	\lim_{\theta \to 0} \frac{\abs{(\ell,\ell+\theta) \cap E}}{\theta}
	=
	1
	\qquad
	(\text{resp.\ }
	\lim_{\theta \to 0} \frac{\abs{(\ell-\theta,\ell+\theta) \cap E}}{2\theta}
	=
	1)
	.
\end{equation}
It is easy to see that every density point is also a right density point, cf.\ \cite[p.\ 32]{Fattorini-05}, and it follows from
Lebesgue's differentiation theorem that almost every point $\ell \in E$ is a density point of $E$, and thus a right density point,
see, e.g., \cite[Corollary~2.14]{Mattila-99}.

The following result is an adapted version of \cite[Proposition~2.1]{PhungW-13}; see also \cite[Lemma~2.1.5]{Fattorini-05} for a
similar statement. We give a proof here in order to be self-contained.

\begin{proposition}\label{prop:phungw-13}
	Let $E \subset \RR$ be measurable with positive Lebesgue measure, and let $\ell \in \RR$ be a right density point of $E$. Then,
	given $q \in (0,1)$, for every $\ell_1 > \ell$ sufficiently close to $\ell$ the strictly decreasing sequence
	$(\ell_m)_{m \in \NN}$ with
	\begin{equation}\label{eq:defellm}
		\ell_m
		=
		\ell + q^{m-1}(\ell_1 - \ell)
		,\quad
		m \in \NN
		,
	\end{equation}
	satisfies
	\begin{equation}\label{eq:ellmEstimate}
    \ell_m -  \ell_{m + 1}
    \leq
    3 \abs{ (\ell_{m + 1}, \ell_m) \cap E }
    ,\quad
    m \in \NN
    .
  \end{equation}
\end{proposition}

\begin{proof}
	Since $\ell$ is a right density point of $E$ by hypothesis, there is by \eqref{eq:lebesguePoint} some $\theta_0 > 0$, depending
	on $E$ and $q$, such that
	\begin{equation}\label{eq:choicetheta}
		\abs{ (\ell, \ell + \theta) \setminus E }
		<
		\frac{1 - q}{2(1 + q)} \abs{ (\ell, \ell + \theta) \cap E }
		\quad\text{ for all }\
		\theta < \theta_0
		.
	\end{equation}
	We fix an arbitrary $\ell_1 \in (\ell,\ell+\theta_0)$ and define the sequence $(\ell_m)_{m\in\NN}$ as in \eqref{eq:defellm}. It
	remains to show that \eqref{eq:ellmEstimate} holds. To this end, let $m \in \NN$, and set $\theta \coloneqq \ell_m - \ell$, so
	that $\ell + \theta = \ell_m$ and $\theta < \ell_1 - \ell < \theta_0$. Inserting this into \eqref{eq:choicetheta} gives
	\begin{equation}\label{eq:estimateEComplement}
  	\abs{ (\ell_{m + 1}, \ell_m ) \setminus E }
  	\leq
  	\abs{ (\ell, \ell_m) \setminus E }
		<
		\frac{1 - q}{2(1 + q)} \abs{ (\ell, \ell_m) \cap E }
		.
	\end{equation}
	In order to bound the right-hand side further, we observe that
	\begin{align*}
		\ell_{m+1} - \ell
		&=
		q^m(\ell_1 - \ell)
			=
			\frac{q}{1-q} \, (1 - q) q^{m - 1} (\ell_1 - \ell)\\
		&=
		\frac{q}{1 - q} \, (q^{m - 1} - q^m) (\ell_1 - \ell) 
			=
			\frac{q}{1 - q} \, (\ell_m - \ell_{m + 1})
		.
	\end{align*}
	This leads to the estimate
	\begin{align*}
		\abs{ (\ell, \ell_m) \cap E }
		&\leq
		\abs{ (\ell, \ell_{m+1}) } + \abs{ (\ell_{m+1},\ell_m) \cap E }\\
		&\leq
		\frac{q}{1-q} (\ell_m - \ell_{m+1}) + \abs{ (\ell_{m+1},\ell_m) \cap E }
		,
	\end{align*}
	which together with \eqref{eq:estimateEComplement} and the bounds $q/(1+q) < 1$ and $(1-q)/(1+q) < 1$ yields
	\[
    \abs{ (\ell_{m + 1}, \ell_m ) \setminus E }
    <
    \frac{1}{2} \Bigl( \abs{ (\ell_{m + 1}, \ell_m) \cap E } + (\ell_m - \ell_{m + 1}) \Bigr)
    .
	\]
	From this we conclude that
	\begin{align*}
    \ell_m - \ell_{m + 1}
		=
    \abs{ (\ell_{m + 1}, \ell_m ) \cap E } + \abs{ (\ell_{m + 1}, \ell_m ) \setminus E }
    \leq
    \frac{3}{2} \abs{ (\ell_{m + 1}, \ell_m) \cap E } + \frac{1}{2} \bigl( \ell_{m} - \ell_{m + 1} \bigr)
    ,
	\end{align*}
	and rearranging terms shows the desired estimate \eqref{eq:ellmEstimate}. This completes the proof.
\end{proof}%

\begin{remark}\label{rm:phungw-13}
	The proof of Proposition~\ref{prop:phungw-13} shows that \eqref{eq:choicetheta} can actually be relaxed to
	\[
		\abs{ (\ell, \ell + \theta) \setminus E }
		<
		\min\Bigl\{ \frac{1 - q}{2q} , \frac{1}{2} \Bigr\} \, \abs{ (\ell, \ell + \theta) \cap E }
		\quad\text{ for all }\
		\theta < \theta_0
	\]
	with the corresponding modification to \eqref{eq:estimateEComplement}. The rest of the proof then carries over.
\end{remark}

\section*{Acknowledgements}
The first author is grateful to Christian Seifert for the supervision of his Ph.D.\ project in which a similar version of
Theorem~\ref{thm:observability} was proven, cf.~\cite{Gabel_diss}. The second author has been partially supported by the DFG
grant VE 253/10-1 entitled \textit{Quantitative unique continuation properties of elliptic PDEs with variable 2nd order
coefficients and applications in control theory, Anderson localization, and photonics}.

%%%%%%%%%%%%%%%%%%%%%%%%%%%%%%%%%%%%%%%%%%%%%%%%%%%%%%%%%%%%%%%%%%%%%%%%%%%%%%%%%%%%%%%%%%%%%%%%%%%%%%%%%%%%%%%%%%%%%%%%%%%%%%%%%%%%%%%%%%%%%%%%%%%%%%%%%%%%%%%%%%%%

\end{document}